\documentclass{amsart}

\usepackage{amssymb,amsmath,amsthm,amscd}
\usepackage{tikz}
\usepackage{amscd}
\usepackage{mathrsfs}
\usepackage{color}
\usepackage[all]{xy}
\usepackage{comment}

\theoremstyle{definition}
\newtheorem{thm}{Theorem}[section]

\newtheorem{prp}[thm]{Proposition}
\newtheorem{dfn}[thm]{Definition}
\newtheorem{cor}[thm]{Corollary}

\newtheorem{rmk}[thm]{Remark}

\newtheorem{eg}[thm]{Example}

\theoremstyle{case}

\title[Compact and Schatten-Class Perturbations of Partial Isometries]{Spectral Characterizations of Compact and Schatten-Class Perturbations of Partial Isometries}

\author[Neeru Bala]{Neeru Bala}
\address{Department of Mathematics and Computing\\ Indian Institute of Technology (ISM) Dhanbad\\ Dhanbad, Jharkhand\\ India, 826 004.}
\email[]{neerubala@iitism.ac.in}

\author[Ramesh Golla]{Ramesh Golla}
\address{Department of Mathematics\\Indian Institute of Technology Hyderabad, Kandi\\ Sangareddy, Telangana \\ India, 502 284.}
\email{rameshg@math.iith.ac.in}

\date{\today}



\begin{document}
\begin{abstract}
 We characterize bounded operators that are compact (respectively, Schatten-class) perturbations of scalar multiples of partial isometries with finite-dimensional kernel. Our characterizations are formulated in terms of the essential spectrum of $T^*T$, absolutely norm attaining operators, and the Moore--Penrose inverse. In particular, we show that an operator $T$ is a Schatten-class perturbation of a partial isometry with finite-dimensional kernel if and only if $\sigma_{\mathrm{ess}}(T^*T)$ is a singleton and the discrete spectrum of $T^*T$ satisfies a corresponding $\ell^p$-summability condition. We further obtain equivalent criteria involving the compactness (or Schatten-class membership) of $\alpha I-T^*T$ and $\alpha T^\dagger-T^*$.

As applications, we establish characterizations of compact and Schatten-class perturbations of isometries, describe the corresponding behavior of Moore--Penrose inverses, and derive factorization results for closed-range operators. In particular, we provide a new Moore--Penrose inverse proof of a theorem of Şerban and Turcu and obtain an explicit formula for the factorizing operator.

\end{abstract}

\maketitle

\section{Introduction and Preliminaries}

Partial isometries and their perturbations play an important role in operator theory and arise naturally in Fredholm theory, dilation theory, and the study of essentially isometric operators. Consequently, compact perturbations of partial isometries have been extensively investigated from both spectral and geometric viewpoints (see \cite{SerbanTurcu,IoanaTurcu,dassarkar} and the references therein). More recently, these investigations have revealed close connections with absolutely norm attaining operators and absolutely minimum attaining operators \cite{GRSSS1,ANHISP}.

A bounded operator $T\in\mathcal B(H)$ is called \emph{absolutely norm attaining} if, for every non-zero closed subspace $M$ of $H$, the restriction $T|_M$ attains its norm. We refer the reader to \cite{CarvajalNeves1,PandePaulsen,GRVN,GRpara,ANHISP,GRSSS1} for background on this class of operators. A fundamental result of \cite{GRSSS1} shows that the operator norm closure of $\mathcal {AN}(H)$ coincides with the class of operators whose modulus has singleton essential spectrum. Since operators with singleton essential spectrum are precisely compact perturbations of scalar multiples of partial isometries with finite-dimensional kernel (\cite[Theorem 4.13]{ANHISP}), this establishes a close relationship between absolutely norm attaining operators and compact perturbations of partial isometries. It is therefore natural to ask whether analogous characterizations exist for Schatten-class perturbations.

The primary objective of this paper is to answer this question. We establish a complete spectral characterization of Schatten-class perturbations of scalar multiples of partial isometries with finite-dimensional kernel. More precisely, we prove that an operator $T\in\mathcal B(H)$ is a Schatten $p$-class ($1\leq p<\infty$) perturbation of a scalar multiple of a partial isometry with finite-dimensional kernel if and only if the essential spectrum of $T^{*}T$ consists of a single point and the corresponding discrete spectrum satisfies a natural $\ell^p$-summability condition. Thus, a Schatten-class perturbation of a scalar multiple of a partial isometry can be recognized entirely from the spectral behaviour of $T^{*}T$, providing a natural analogue of the corresponding characterization for compact perturbations.

The main theorem admits several equivalent operator-theoretic formulations. In particular, we prove that the above spectral characterization is equivalent to each of the conditions
\[
\alpha I-T^{*}T\in\mathcal K_p(H), \qquad 1\le p<\infty,
\]
and
\[
\alpha T^\dagger-T^{*}\in\mathcal K_p(H), \qquad 1\le p<\infty,
\]
where $T^\dagger$ denotes the Moore--Penrose inverse of $T$. These equivalences provide new criteria for recognizing Schatten-class perturbations of partial isometries and establish direct connections among perturbation theory, spectral theory, generalized inverses, semi-Fredholm operators and absolutely norm attaining operators.

The characterization developed in this paper has several applications. We obtain analogous characterizations for compact and finite-rank perturbations of partial isometries and isometries, derive criteria involving semi-Fredholm operators and the Moore--Penrose inverse, and present a new proof of a factorization theorem of \cite{SerbanTurcu} together with an explicit representation of the factorizing operator. Furthermore, the main theorem shows that the spectral hypothesis appearing in our earlier work \cite{ANHISP} is equivalent to an operator being a Schatten-class perturbation of a scalar multiple of a partial isometry with finite-dimensional kernel. Consequently, every operator satisfying the equivalent conditions of the main theorem belongs to a class for which the existence of a non-trivial hyperinvariant subspace is already known.

The paper is organized as follows. In Section~2, we establish the characterization theorem together with several equivalent formulations and applications to compact and finite-rank perturbations. Section~3 is devoted to Moore--Penrose inverses, perturbations of isometries, and factorization results.

\subsection{Preliminaries}
Throughout the paper, $H$ denotes an infinite-dimensional complex separable Hilbert space and $\mathcal B(H)$ denotes the $C^{*}$-algebra of all bounded linear operators on $H$. For $T\in\mathcal B(H)$, we write $R(T)$ and $N(T)$ for the range and null space of $T$, respectively, and
\[
|T|=(T^{*}T)^{1/2}
\]
for the modulus of $T$. The spaces of finite-rank, compact and Schatten $p$-class ($1 \leq p<\infty$) operators are denoted by $\mathcal F(H)$, $\mathcal K(H)$ and $\mathcal K_p(H)$, respectively.

For $T\in\mathcal  B(H)$, we define the \textit{spectrum} of $T$ by $$\sigma(T):=\{\lambda \in \mathbb{C} : T-\lambda I \text{ is not invertible}\}.$$ It is a non-empty compact set in the complex plane. An operator $T\in\mathcal  B(H)$ is called \textit{Fredholm} if $R(T)$ is closed and $N(T)$, $N(T^*)$ are finite-dimensional. The \textit{essential spectrum} of $T$ is defined by $$\sigma_{\text{ess}}(T):=\{\mu \in \mathbb{C} : T-\mu I \text{ is not Fredholm}\}.$$ 

The quotient algebra $\mathcal  B(H)/\mathcal  K(H)$ is called the \textit{Calkin algebra} and we denote the canonical homomorphism from $\mathcal  B(H)$ onto $\mathcal  B(H)/\mathcal  K(H)$ by $\pi$. As a consequence of Atkinson's characterization for Fredholm operators, we have $\sigma_{\text{ess}}(T)=\sigma(\pi(T))$. Note that $\sigma_{\text{ess}}(T)$ is also a non-empty compact subset of $\sigma(T)$. If $T\in \mathcal B(H)$ is self-adjoint, then \cite[Theorem VII.11, Page 236]{reedsimon1}
\[
\sigma_{\mathrm{ess}}(T)
=
\sigma(T)\setminus
\{\lambda:\lambda \text{ is an isolated eigenvalue of finite multiplicity}\}.
\]
Equivalently, a point of $\sigma(T)$ belongs to $\sigma_{\mathrm{ess}}(T)$ if it is an accumulation point of $\sigma(T)$ or an eigenvalue of infinite multiplicity.

If $A \in \mathcal  B(H)$ is self-adjoint and $K\in \mathcal  K(H)$ is also self-adjoint, then by Weyl's theorem \cite[Corollary 8.16, Page 182]{Schmudgen}, we have
$$\sigma_{\text{ess}}(A+K)=\sigma_{\text{ess}}(A).$$ Also, we refer to  \cite[Theorem 2]{Kover} regarding Weyl's theorem.

 The singular values $s_n(A)$ of $A\in\mathcal K(H)$ are the eigenvalues of $|A|=(A^*A)^{1/2}$, listed in non-increasing order and repeated according to multiplicity. We say a compact operator $A$ is in the Schatten $p$-class $\mathcal K_p(H), (1 \leq p<\infty )$, if
 $\sum_{n=1}^{\infty} s_n(A)^p <\infty$.
 Note that $\mathcal F(H)\subseteq\mathcal K_p(H)\subseteq\mathcal K(H)$ for $1\leq p<\infty$.
For a self-adjoint operator $A$, we write $\sigma_d(A)$ for the set of isolated eigenvalues of finite multiplicity. Whenever an $\ell^p$-condition is imposed on the discrete eigenvalues, they are understood to be repeated according to multiplicity. We also use the essential minimum modulus
\[
m_e(T):=\inf\sigma_{\mathrm{ess}}(|T|).
\]

 The principal theorem of this paper establishes several equivalent
characterizations of Schatten $p$-class perturbations of scalar multiples of
partial isometries with finite-dimensional kernel. In particular, we show
that such perturbations can be characterized entirely in terms of the
spectral behaviour of $T^*T$. Furthermore, this spectral characterization is
equivalent to the operator-theoretic conditions
\[
\alpha I-T^*T\in\mathcal K_p(H)
\quad\text{and}\quad
\alpha T^\dagger-T^*\in\mathcal K_p(H)\;(1\leq p<\infty),
\]
thus revealing close connections between perturbation theory, spectral
theory, generalized inverses and absolutely norm attaining operators. The
precise statement of the main theorem is given in
Theorem~\ref{schattenpluspi}.
\section{Compact and Schatten-Class Perturbations of Partial Isometries}
In this section, we prove our main results. First, we make a few important observations regarding compact and Schatten $p$-class ($1\leq p<\infty$) perturbations of partial isometries with finite-dimensional kernel. These observations lead to several interesting consequences.
\begin{thm}\label{coptperturbpi}
Let $V\in\mathcal B(H)$ be a partial isometry with $\dim N(V)<\infty$, let $K\in\mathcal K(H)$, and let $\alpha>0$. If
\[
T=\alpha V+K,
\]
then $R(T)$ is closed and $N(T)$ is finite-dimensional.
\end{thm}
\begin{proof}
Since $V^*V=P_{N(V)^\perp}$,
\begin{align}
T^*T
&=K^*K+\alpha(K^*V+V^*K)+\alpha^2P_{N(V)^\perp} \nonumber\\
&=\alpha^2I+\widetilde K, \label{cptplusid}
\end{align}
where
\begin{equation}\label{tildacpt}
\widetilde K=K^*K+\alpha(K^*V+V^*K)-\alpha^2P_{N(V)}.
\end{equation}
Because $N(V)$ is finite-dimensional, $P_{N(V)}$ is finite-rank; hence $\widetilde K$ is compact and self-adjoint. Thus $T^*T=\alpha^2I+\widetilde K$ is Fredholm, so $R(T^*T)$, and therefore $R(T)$, is closed.

If $x\in N(T)=N(T^*T)$, then \eqref{cptplusid} gives $\alpha^2x=-\widetilde Kx$. Hence
\[
\alpha^2 I_{N(T)}=-\widetilde K|_{N(T)}.
\]
Thus the identity on $N(T)$ is compact, which is possible only when $N(T)$ is finite-dimensional.
\end{proof}
 \begin{rmk}\label{remark}
    Let $V$ be a partial isometry with finite-dimensional null space and $T=\alpha V+K$ for some compact operator $K$. Then $T=\alpha S+\tilde{K}$, where either $S$ is an isometry or a co-isometry and $\tilde{K}$ is a compact operator.
\end{rmk}
\begin{proof}
    If $\dim N(V)\leq \dim N(V^*)$, then there exist an isometry $W:N(V)\rightarrow N(V^*)$. Extend $W$ by zero on $N(V)^\perp$; then $W$ is a partial isometry on $H$. We can write $T=\alpha(V+W)+\tilde{K}$, where $V+W$ is an isometry and $\tilde{K}$ is a compact operator.

    On the other hand, if $\dim N(V^*)\leq\dim N(V)$, then there exist an isometry $\tilde{W}:N(V^*)\rightarrow N(V)$ and $T^*=\alpha (V^*+\tilde{W})+\hat{K},$ where $V^*+\tilde{W}$ is an isometry. Taking adjoints gives the required representation of $T$ as a compact perturbation of a co-isometry.
\end{proof}

\begin{rmk}
    By Theorem 3.1 of \cite{Brown}, we know that for $\alpha>0$ we have $\sigma_{\text{ess}}(T^*T)=\{\alpha^2\}=\sigma_{\mathrm{ess}}(TT^*)$ if and only if $T=\alpha S+K$, where $S$ is a unitary or an isometry  or a co-isometry with finite-dimensional null space.
\end{rmk}
The following characterization theorem plays a central role in the proofs of the main results.
\begin{thm}\label{ANcloschar}\cite[Theorem 4.13]{ANHISP}
Let $T \in \mathcal{B}(H)\setminus \mathcal{K}(H)$. Then the following statements are equivalent:

\begin{enumerate}
\item $\sigma_{\mathrm{ess}}(T^{*}T)=\{\alpha\}$ for some $\alpha>0$;

\item $T^{*}T \in \overline{\mathcal{AN}(H)}$;

\item There exist a partial isometry $V$ with $\dim N(V)<\infty$ and a compact operator
      $K\in \mathcal{K}(H)$ such that
      \[
      T = \sqrt{\alpha}\,V + K.
      \]
\end{enumerate}
\end{thm}
Next, we prove the corresponding result with compact operators replaced by Schatten $p$-class operators.
\begin{thm}\label{essschattenthm}
Let $T\in\mathcal B(H)\setminus\mathcal K(H)$, let $1\leq p<\infty$, and let $\alpha>0$. Then the following are equivalent:
\begin{enumerate}
\item\label{schattenclassperturb} There exist $K\in\mathcal K_p(H)$ and a partial isometry $V$ with $\dim N(V)<\infty$ such that
\[
T=\alpha V+K.
\]
\item\label{elpcondn} $\sigma_{\mathrm{ess}}(T^*T)=\{\alpha^2\}$ and, if $(\lambda_n)$ denotes the discrete eigenvalues of $T^*T$ different from $\alpha^2$, repeated according to multiplicity, then
\[
\sum_n |\lambda_n-\alpha^2|^p<\infty.
\]
\end{enumerate}
\end{thm}
\begin{proof}
Assume \eqref{schattenclassperturb}. By \eqref{cptplusid},
\[
T^*T=\alpha^2I+\widetilde K,
\]
where $\widetilde K\in\mathcal K_p(H)$; indeed, $\mathcal K_p(H)$ is a two-sided ideal, $K^*K\in\mathcal K_p(H)$, and $P_{N(V)}$ is finite-rank. Weyl's theorem gives
\[
\sigma_{\mathrm{ess}}(T^*T)=\{\alpha^2\}.
\]
Since $\widetilde K$ is compact and self-adjoint, its nonzero eigenvalues $(\mu_n)$, repeated according to multiplicity, satisfy $\sum_n|\mu_n|^p<\infty$. The discrete eigenvalues of $T^*T$ different from $\alpha^2$ are precisely $\lambda_n=\alpha^2+\mu_n$, with the same multiplicities. Hence
\[
\sum_n|\lambda_n-\alpha^2|^p=\sum_n|\mu_n|^p<\infty.
\]

Conversely, assume \eqref{elpcondn}. The self-adjoint operator
\[
A:=T^*T-\alpha^2I
\]
has essential spectrum $\{0\}$ and its nonzero eigenvalues, repeated according to multiplicity, are $(\lambda_n-\alpha^2)$. The assumed $\ell^p$-summability therefore implies $A\in\mathcal K_p(H)$. Since
\[
(|T|-\alpha I)(|T|+\alpha I)=T^*T-\alpha^2I=A
\]
and $|T|+\alpha I$ is invertible, we obtain
\[
|T|-\alpha I=A(|T|+\alpha I)^{-1}\in\mathcal K_p(H).
\]
Let $T=V|T|$ be the polar decomposition. Then
\[
T-\alpha V=V(|T|-\alpha I)\in\mathcal K_p(H).
\]
Moreover, $T^*T=\alpha^2I+A$ is Fredholm, so $N(T)=N(T^*T)=N(V)$ is finite-dimensional. Thus \eqref{schattenclassperturb} holds.
\end{proof}

Next, we establish several equivalent conditions for compact (Schatten-class) perturbations of partial isometries with finite-dimensional kernels. The following proposition motivates the subsequent characterization theorems.
\begin{prp}\cite[Proposition 4]{dassarkar}
Let $T\in \mathcal B(H)$ be left-invertible and of finite index. The following
statements are equivalent.
\begin{enumerate}
\item  $T = \text{compact} + \text{isometry}$.
\item  $I - T^{*}T$ is compact.
\item  $L_T -T^{*}$ is compact, where $L_{T}:=(T^*T)^{-1}T^*$.
\item $I - TT^{*}$ is compact.
\end{enumerate}
\end{prp}

First, we prove the result for Schatten-class perturbations. Here, we recall the definition of the Moore–Penrose inverse.
\begin{dfn}\cite[Definition P, Page 48]{GroetschGI}
If $T \in \mathcal B(H_1, H_2)$ has closed range, then
$T^{\dagger} $ is the unique operator in $\mathcal B(H_2, H_1)$ satisfying
\begin{enumerate}
\item $TT^{\dagger} = (TT^{\dagger})^*$
\item $T^{\dagger}T = (T^{\dagger}T)^{*}$
\item  $TT^{\dagger}T = T$ 
\item  $T^{\dagger}TT^{\dagger} = T^{\dagger}$.
\end{enumerate}
\end{dfn}
The Moore–Penrose inverse can be defined for operators that are not necessarily bounded or do not have closed range. For more details, we refer to \cite{BenGre}.
\begin{prp}\label{mpirepn}\cite{GroetschGI}
Let $T\in \mathcal  B(H_1,H_2)$ has a closed range. Then
\begin{equation*}
T^{\dagger}=T^*(TT^*)^{\dagger}=(T^*T)^{\dagger}T^*.
\end{equation*}
\end{prp}

We now arrive at the main result of this paper, which provides a spectral characterization of Schatten-class perturbations of scalar multiples of partial isometries. All subsequent characterizations and applications will be derived from this theorem.
\begin{thm}[Main Theorem]\label{schattenpluspi}
Let $T\in\mathcal B(H)$, let $1\leq p<\infty$, and let $\alpha>0$. Then the following are equivalent:
\begin{enumerate}
\item\label{cond1} $\sigma_{\mathrm{ess}}(T^*T)=\{\alpha\}$ and, if $(\lambda_n)$ denotes the discrete eigenvalues of $T^*T$ different from $\alpha$, repeated according to multiplicity, then
\[
\sum_n|\lambda_n-\alpha|^p<\infty.
\]
\item\label{cond2} There exist a partial isometry $V$ with $\dim N(V)<\infty$ and $K\in\mathcal K_p(H)$ such that
\[
T=\sqrt\alpha\,V+K.
\]
\item\label{cond3} $\alpha I-T^*T\in\mathcal K_p(H)$.
\item\label{cond4} $T$ has closed range, $\alpha T^\dagger-T^*\in\mathcal K_p(H)$, and $N(T)$ is finite-dimensional.
\end{enumerate}
\end{thm}
\begin{proof}
The equivalence of \eqref{cond1} and \eqref{cond2} follows from Theorem~\ref{essschattenthm}, with $\sqrt\alpha$ in place of $\alpha$.

Assume \eqref{cond2}. Expanding $T^*T$ gives
\[
T^*T-\alpha I
=\sqrt\alpha\,(V^*K+K^*V)+K^*K-\alpha P_{N(V)}\in\mathcal K_p(H),
\]
so \eqref{cond3} holds.

Assume \eqref{cond3}. Then $T^*T=\alpha I+K$ for some $K\in\mathcal K_p(H)$. Hence $T^*T$ is Fredholm, so $R(T)$ is closed and $T^\dagger\in\mathcal B(H)$. Since $R(T)$ is closed,
\[
(T^*T)T^\dagger=T^*.
\]
Therefore
\[
(\alpha I-T^*T)T^\dagger=\alpha T^\dagger-T^*\in\mathcal K_p(H).
\]
Moreover, $N(T)=N(T^*T)$ is finite-dimensional because $T^*T$ is Fredholm. Thus \eqref{cond4} holds.

Finally, assume \eqref{cond4}. Then
\[
(\alpha T^\dagger-T^*)T
=\alpha P_{N(T)^\perp}-T^*T\in\mathcal K_p(H).
\]
Since $N(T)$ is finite-dimensional,
\[
\alpha I-T^*T
=(\alpha T^\dagger-T^*)T+\alpha P_{N(T)}\in\mathcal K_p(H).
\]
Thus \eqref{cond3} holds, completing the proof.
\end{proof}
\begin{eg}
Let $T$ be the unilateral weighted shift on $\ell^2(\mathbb N)$ defined by
\[
Te_n=w_ne_{n+1}, \qquad n\ge 1,
\]
where $\{e_n:n\in\mathbb N\}$ is the standard orthonormal basis of $\ell^2(\mathbb N)$ and $(w_n)\in \ell^{\infty}(\mathbb N)$ with $w_n>0$ for each $n\in \mathbb N$.
Assume that $w_n\to \alpha>0$ and $(w_n-\alpha)\in \ell^p.$

Let $S$ denote the unilateral shift. Then $T=\alpha S+K,$ where 
$$Ke_n=(w_n-\alpha)e_{n+1}.$$

Since $\sum_{n=1}^{\infty}|w_n-\alpha|^p<\infty,$ it follows that $K\in \mathcal K_p(H)$. Indeed, $$K^*Ke_n=|w_n-\alpha|^2e_n,$$
and hence the singular values of $K$ are given by
\[
s_n(K)=|w_n-\alpha|.
\]
Therefore,
\[
\sum_{n=1}^{\infty}s_n(K)^p
=
\sum_{n=1}^{\infty}|w_n-\alpha|^p
<\infty.
\]

Thus $T$ is a Schatten $p$-class perturbation of the isometry $\alpha S$.

Since $(\alpha S)^*(\alpha S)=\alpha^2 I,$ we obtain $\sigma_{\mathrm{ess}}(T^*T)=\{\alpha^2\}.$ Moreover, $m_e(T)=\alpha>0,$ and hence $T$ is left semi-Fredholm.

This example illustrates Theorem \ref{schattenpluspi}. As a concrete example, the weighted shift with weights
\[
w_n=\alpha+\frac{1}{(n+1)^r},
\qquad r>\frac1p,
\]
satisfies the above hypotheses.
\end{eg}

\begin{rmk}
\begin{enumerate}
\item   Multiplying $\alpha T^{\dagger}-T^* \in \mathcal  K_{p}(H)$, on the left by $T$, we obtain that $\alpha T^{\dagger}-T^* \in \mathcal K_p(H)$ implies $\alpha P_{R(T)}-TT^*\in \mathcal  K_{p}(H)$. However, it is not clear whether the converse holds.
\item It is to be noted that the operator $T$ satisfying the condition (\ref{cond2}) of Theorem~\ref{schattenpluspi} has a non-trivial hyperinvariant subspace. For more details, we refer to \cite[Corollary 3.4]{ANHISP}.
\end{enumerate}
\end{rmk} 
\begin{cor}
Let $T\in \mathcal{B}(H)$. If
\[
\alpha I-T^{*}T\in\mathcal K_p(H)
\]
for some $\alpha>0$ and $1\le p<\infty$, then $T$ has a nontrivial hyperinvariant subspace.

Equivalently, if
\[
\alpha T^{\dagger}-T^{*}\in\mathcal K_p(H)
\]
and $\dim N(T)<\infty$, then $T$ has a nontrivial hyperinvariant subspace.
\end{cor}

\begin{proof}
The conclusion follows immediately from Theorem \ref{schattenpluspi} together with Corollary 3.4 of \cite{ANHISP}.
\end{proof}
Following similar steps as in Theorem \ref{schattenpluspi}, with the help of \cite{ANHISP}, we can prove the following two results.
\begin{thm}\label{essfrankthm}
Let $T\in \mathcal  B(H)$. Then the following are equivalent:
\begin{enumerate}
    \item\label{frankperturb} there exist $F\in \mathcal F(H)$, $\alpha>0$, and a partial isometry $V$ with $\dim N(V)<\infty$ such that $T=F+\alpha V$
    \item\label{finiteeigcondn} $\sigma_{\text{ess}}(T^*T)={\{\alpha^2}\}$ and $T^*T$ has only finitely many discrete eigenvalues different from $\alpha^2$, counted with multiplicity.
\end{enumerate}
\end{thm}

\begin{thm}
Let $T\in \mathcal  B(H)\setminus \mathcal K(H)$. Then the following are equivalent:
\begin{enumerate}

\item \label{ANcloschar1}  $T \in \overline{\mathcal  {AN}(H)}^{\|\cdot \|}$, (the closure is with respect to the  operator norm)
\item \label{ANcloschar2} $\sigma_{\text{ess}}(T^*T)={\{\alpha}\}$ for some $\alpha > 0$ 
\item \label{ANcloschar3} there exist $K\in \mathcal K(H)$ and a partial isometry $V$ with $\dim N(V)<\infty$ such that $T=K+\sqrt{\alpha}V$
 
\item \label{ANcloschar4}  $\alpha I-T^*T\in \mathcal  K(H)$ 
\item \label{ANcloschar5} $\alpha T^{\dagger}-T^* \in \mathcal  K(H)$ and $N(T)$ is finite-dimensional
\end{enumerate}
\end{thm}
\begin{proof}
The equivalence of (\ref{ANcloschar1}), (\ref{ANcloschar2}) and (\ref{ANcloschar3}) follows from \cite[Theorem 4.13]{ANHISP}. Next, assume that  (\ref{ANcloschar4}) is true. Then $(\alpha I-T^*T)T^{\dagger}\in \mathcal K(H)$, that is, 
\begin{equation*}
\alpha T^{\dagger}-T^*TT^{\dagger}=\alpha T^{\dagger}-T^*P_{R(T)}=\alpha T^{\dagger}-T^*(I-P_{N(T^*)})=\alpha T^{\dagger}-T^*\in \mathcal K(H).
\end{equation*}
The fact that $N(T)$ is finite-dimensional is clear. The implication (\ref{ANcloschar3}) $\Rightarrow$ (\ref{ANcloschar4}) can be obtained by looking at $T^*T$.

 If (\ref{ANcloschar5}) is true then we have $(\alpha T^{\dagger}-T^*)T\in \mathcal K(H)$ or $\alpha P_{N(T)^\perp}-T^*T=K$ or $\alpha I-T^*T=K+\alpha P_{N(T)}\in \mathcal K(H)$, since $N(T)$ is finite-dimensional. This implies that $T^*T \in \overline{\mathcal  {AN}(H)}^{\|\cdot \|}$ by \cite[Theorem 3.15]{GRSSS1}. But this in turn implies $T \in \overline{\mathcal  {AN}(H)}^{\|\cdot \|}$.
\end{proof}
The following result can be proved very easily using the arguments as in Theorems \ref{essschattenthm} and \ref{schattenpluspi}.

\begin{thm}
Let $T\in \mathcal  B(H)$. The following are equivalent:
\begin{enumerate}
\item $\sigma_{\text{ess}}(T^*T)={\{\alpha}\}$ for some $\alpha > 0$ and $T^*T$ has only finitely many discrete eigenvalues different from $\alpha$, counted with multiplicity
\item there exist $F\in \mathcal F(H)$ and a partial isometry $V$ with $\dim N(V)<\infty$ such that $T=F+\sqrt{\alpha}V$
\item $\alpha I-T^*T\in \mathcal  F(H)$ 
\item $\alpha T^{\dagger}-T^* \in \mathcal  F(H)$ and $N(T)$ is finite-dimensional.
\end{enumerate}
\end{thm}

Next, we illustrate our main theorem with an example.
\begin{eg}
Let $T=\operatorname{diag}\left\{c+\frac{1}{(n+1)^r}\right\}_{n\ge1}$
on $\ell^2(\mathbb N)$, where $c>0$ and $r>\frac1p$. Then
$T=c I+K,$
where
\[
Ke_n=\frac{1}{(n+1)^r}e_n,\qquad n\ge1.
\]
Since $\sum_{n=1}^{\infty}\frac1{(n+1)^{rp}}<\infty,$ we have $K\in\mathcal K_p(H)$.

Moreover, $T$ is positive and invertible, so $T^\dagger=T^{-1}$. Therefore,
\[
(c^2T^\dagger-T^*)e_n
=
\left(
\frac{c^2}{c+\frac1{(n+1)^r}}
-c
-\frac1{(n+1)^r}
\right)e_n,
\]
for every $n\ge1$. Since
\[
\frac{c^2}{c+t}-c-t
=
-\left(1+\frac{c}{c+t}\right)t,
\]
it follows that
\[
\left|
\frac{c^2}{c+\frac1{(n+1)^r}}
-c
-\frac1{(n+1)^r}
\right|
\le
\frac{2}{(n+1)^r},
\]
for all sufficiently large $n$. Hence $c^2T^\dagger-T^*\in\mathcal K_p(H)$,
which illustrates Theorem \ref{schattenpluspi} (with spectral parameter $c^2$).
\end{eg}

\section{Moore--Penrose Inverses, Isometries, and Factorization Results}
We first describe the behavior of the Moore--Penrose inverse under Schatten-class perturbations of scalar multiples of partial isometries.
\begin{thm}\label{mpperturb}
Let $T\in \mathcal{B}(H)$ and $1\le p<\infty$. Assume that $T=\sqrt{\alpha}\,V+K,$
where $K\in \mathcal K_p(H)$, $\alpha>0$, and $V$ is a partial isometry with
$\dim N(V)<\infty$. Then $T$ has closed range and
\[
T^\dagger
=
\frac{1}{\sqrt{\alpha}}\,V^*+L,
\]
for some $L\in \mathcal K_p(H)$.

In particular,
\[
T^\dagger-\frac{1}{\sqrt{\alpha}}V^*
\in \mathcal K_p(H).
\]
\end{thm}

\begin{proof}
Since $T=\sqrt{\alpha}\,V+K,\; K\in \mathcal K_p(H),$ Theorem~\ref{coptperturbpi} implies that $T$ has closed range. Hence $T^\dagger\in B(H)$. By Theorem~\ref{schattenpluspi}, we have $\alpha T^\dagger-T^*\in \mathcal K_p(H).$ Since $T^*=\sqrt{\alpha}\,V^*+K^*,$ and $K^*\in \mathcal K_p(H)$, it follows that
\[
\alpha T^\dagger-\sqrt{\alpha}\,V^*
=
(\alpha T^\dagger-T^*)
+
K^*
\in \mathcal K_p(H).
\]

Dividing by $\alpha$, we obtain $T^\dagger-\frac{1}{\sqrt{\alpha}}V^*
\in \mathcal K_p(H).$ Therefore there exists $L\in \mathcal K_p(H)$ such that $T^\dagger
= \frac{1}{\sqrt{\alpha}}V^*+L.$

This completes the proof.
\end{proof}

\begin{cor}
Let $T\in \mathcal{B}(H)$ and $1\le p<\infty$. Then the following are equivalent.

\begin{enumerate}
\item
There exist $\alpha>0$, $K\in \mathcal K_p(H)$, and a partial isometry $V$ with
$\dim N(V)<\infty$ such that
\[
T=\sqrt{\alpha}\,V+K.
\]

\item
There exist $\alpha>0$, $L\in \mathcal K_p(H)$, and a partial isometry $V$ with
$\dim N(V)<\infty$ such that
\[
T^\dagger
=
\frac{1}{\sqrt{\alpha}}V^*+L.
\]

\item
\[
\alpha T^\dagger-T^*\in \mathcal K_p(H)
\]
and $N(T)$ is finite-dimensional.
\end{enumerate}
\end{cor}

\begin{proof}
The equivalence of (1) and (3) follows from Theorem~\ref{schattenpluspi}, and (1)$\Rightarrow$(2) follows from Theorem~\ref{mpperturb}. For (2)$\Rightarrow$(1), take adjoints to obtain
\[
(T^\dagger)^*=\frac{1}{\sqrt{\alpha}}V+L^*.
\]
Apply Theorem~\ref{mpperturb} to $(T^\dagger)^*$ with spectral parameter $1/\alpha$. Since $\dim N(V)<\infty$, this gives
\[
((T^\dagger)^*)^\dagger=\sqrt{\alpha}\,V^*+K_1
\]
for some $K_1\in\mathcal K_p(H)$. Using $((T^\dagger)^*)^\dagger=((T^\dagger)^\dagger)^*=T^*$ and taking adjoints, we obtain
\[
T=\sqrt{\alpha}\,V+K_1^*,
\]
which is (1).
\end{proof}
\begin{prp}
Let $T \in \mathcal{B}(H)$. Suppose that $T=\alpha V+K,$ where $\alpha>0$, $V$ is an isometry with $\dim N(V^*)<\infty$, and $K$ is compact. Then
$m_e(T)=\alpha$ and $\operatorname{ind}(T-\lambda I)=-\dim N(V^*)\le 0,\; \text{for all}\; |\lambda|<\alpha.$
\end{prp}

\begin{proof}
Since $V^*V=I$, we have
$T^*T=(\alpha V^*+K^*)(\alpha V+K)=\alpha^2I+\alpha(V^*K+K^*V)+K^*K.$
Therefore, $T^*T-\alpha^2I$ is compact. Passing to the Calkin algebra yields
$\pi(T)^*\pi(T)=\alpha^2I.$ Hence $|\pi(T)|=\alpha I.$ It follows that $\sigma_e(|T|)=\{\alpha\},$ and consequently $m_e(T)=\inf \sigma_e(|T|)=\alpha.$

Now let $|\lambda|<\alpha$. Since $T-\lambda I=(\alpha V-\lambda I)+K,$ it suffices to compute the index of $\alpha V-\lambda I$. Observe that
$\alpha V-\lambda I=\alpha\left(V-\frac{\lambda}{\alpha}I\right).$
Since $\left|\frac{\lambda}{\alpha}\right|<1$, the operator $V-\frac{\lambda}\alpha I$ is Fredholm. 

Next, we claim that $\operatorname{ind}\left(V-\frac{\lambda}{\alpha}I\right)=
-\dim N(V^*).$
To this end, since $V$ is an isometry and $|\mu|<1$,
\[
\|(V-\mu I)x\|
\ge (1-|\mu|)\|x\|,
\qquad x\in H.
\]
Hence $\ker(V-\mu I)=\{0\}$.

By the Wold decomposition theorem,
\[
V=U\oplus S_n,
\qquad
n=\dim N(V^*),\]
where $U$ is unitary and $S_n$ is the unilateral shift of multiplicity
$n$. Since $U-\mu I$ is invertible and
$\operatorname{ind}(S_n-\mu I)=-n,$
we obtain $\operatorname{ind}(V-\mu I) =-n=-\dim N(V^*).$
Therefore,
\[
-\dim N((V-\mu I)^*)
=\operatorname{ind}(V-\mu I)
=-\dim N(V^*),
\]
which proves that
\[
\dim N((V-\mu I)^*)
=
\dim N(V^*).\]
Therefore,
$\operatorname{ind}(\alpha V-\lambda I)=-\dim N(V^*)\le 0.$

Since compact perturbations preserve the Fredholm index,
\[
\operatorname{ind}(T-\lambda I)
=
\operatorname{ind}(\alpha V-\lambda I)
=
-\dim N(V^*)
\le 0.
\]
This completes the proof. 
\end{proof}
We recall the following result, which is useful for our further discussion.
\begin{thm}\label{cptisometricperturb}\cite[Theorem 2.4]{IoanaTurcu}
 Let $V\in \mathcal B(H)$ be an isometry and $T \in \mathcal B(H)$ be a contraction. Then $T-V \in \mathcal K(H)$ if and only if there is a contraction $Y \in \mathcal B(H)$  such that $Y - I\in \mathcal K(H)$ and $T = YV$.
\end{thm}

\begin{cor}

Let $T\in \mathcal  B(H)$ be a contraction. Then the following conditions are equivalent:

\begin{enumerate}
\item $T \in \overline{\mathcal  {AN}(H)}^{\|\cdot \|}$
\item\label{anFactorization} There exists a partial isometry $V$ with finite-dimensional kernel, a contraction $Y\in \mathcal B(H)$, and a real number $\alpha\geq 0$ such that  $Y-\alpha I\in \mathcal  K(H)$  and $T=YV$. Also, $Y\in \overline{\mathcal {AN}(H)}^{\|\cdot\|}$.
\end{enumerate}
\end{cor}
\begin{proof}
Assume first that $T\in\overline{\mathcal{AN}(H)}$, and let $T=V|T|$ be its polar decomposition. By Theorem~\ref{ANcloschar}, there is $\alpha\geq0$ such that $|T|-\alpha I\in\mathcal K(H)$ and $\dim N(V)<\infty$. Define
\[
Y:=TV^*+\alpha P_{R(V)^\perp}.
\]
Since $V^*V=P_{N(V)^\perp}$ and $N(T)=N(V)$, we have $YV=T$. Moreover,
\[
Y-\alpha I=(T-\alpha V)V^*\in\mathcal K(H).
\]
Because $T$ is a contraction, $\alpha\leq1$. With respect to the orthogonal decomposition $H=R(V)\oplus R(V)^\perp$, the operator $Y$ acts as $TV^*$ on $R(V)$ and as $\alpha I$ on $R(V)^\perp$; hence $\|Y\|\leq1$.

Conversely, assume \eqref{anFactorization}. Since $Y-\alpha I\in\mathcal K(H)$, \cite[Proposition 3.5]{GRSSS1} gives $Y\in \overline{\mathcal  {AN}(H)}^{\|\cdot \|}$. Also, by \cite[Remark 4.12]{ANHISP}, $V\in \overline{\mathcal  {AN}(H)}^{\|\cdot \|}$. Since $\overline{\mathcal{AN}(H)}$ is closed under products \cite[Theorem 3.12]{GRSSS1}, we conclude that $T=YV\in \overline{\mathcal  {AN}(H)}^{\|\cdot \|}$.
\end{proof}
Next, we look at compact perturbation result for general operators.
\begin{thm}\label{genfactorization}
Let $A,B\in \mathcal B(H)$, assume that $R(A)$ is closed, and suppose that $N(A)\subseteq N(B)$. Then the following statements are equivalent:
\begin{enumerate}
\item\label{compactDiff} $A-B\in \mathcal K(H)$
\item\label{compactFactorization} there exists $X\in \mathcal B(H)$ with $X-I\in \mathcal K(H)$ such that $B=XA$. 
\end{enumerate}
\end{thm}
\begin{proof}
(\ref{compactDiff}) $\Rightarrow$(\ref{compactFactorization}): Since $R(A)$ is closed, $A^{\dagger}$ is bounded. Define $X=BA^{\dagger}+P_{R(A)^{\bot}}$. Then we have $XA=BA^{\dagger}A=B(I-P_{N(A)})=B$ as $N(A)\subseteq N(B)$.  Also, $X-I=BA^{\dagger}-P_{R(A)}=BA^{\dagger}-AA^{\dagger}=(B-A)A^{\dagger}=KA^{\dagger}\in \mathcal K(H)$. 

Conversely, if $B=XA$ and $X-I\in\mathcal K(H)$, then $B-A=(X-I)A\in\mathcal K(H)$. 
\end{proof}
\begin{rmk}
Theorem \ref{genfactorization} was established in \cite{SerbanTurcu}. The proof presented here yields the explicit factorization
\[
X=BA^\dagger+P_{R(A)^\perp},
\]
and provides a direct argument based on the Moore--Penrose inverse. On the other hand, the operator $X$ constructed in \cite[Theorem 4]{SerbanTurcu} satisfies the stronger property
\[
N(X)=A\bigl(R(A^*)\cap N(B)\bigr)
    =A\bigl(N(A)^\perp\cap N(B)\bigr),
\]
whereas for our choice of $X$ we only have
\[
A\bigl(R(A^*)\cap N(B)\bigr)\subseteq N(X).
\]
Thus, the approach in \cite{SerbanTurcu} yields additional information on the kernel of the factorizing operator, while our proof is shorter and avoids the use of dilation theory.
\end{rmk}

Using the same argument as in the proof of Theorem \ref{genfactorization}, we obtain the following Schatten-class analogue.
\begin{thm}
Let $A,B\in \mathcal B(H)$, assume that $R(A)$ is closed, and suppose that $N(A)\subseteq N(B)$. Then the following statements are equivalent:
\begin{enumerate}
\item\label{schattenDiff} $A-B\in \mathcal K_{p}(H),\; 1\leq p<\infty$
\item\label{schattenFactorization} there exists $X\in \mathcal B(H)$ with $X-I\in \mathcal K_p(H),\; 1\leq p<\infty$ such that $B=XA$. 
\end{enumerate}
\end{thm}
\begin{proof}
If $A-B\in\mathcal K_p(H)$, define $X=BA^\dagger+P_{R(A)^\perp}$ as in Theorem~\ref{genfactorization}. Then $B=XA$ and
\[
X-I=(B-A)A^\dagger\in\mathcal K_p(H).
\]
Conversely, if $B=XA$ and $X-I\in\mathcal K_p(H)$, then
\[
B-A=(X-I)A\in\mathcal K_p(H).\qedhere
\] 
\end{proof}
\section*{Data Availability}
Data sharing is not applicable to this article as no datasets were generated or analyzed during the current study.

\section*{Declaration of Generative AI and AI-Assisted Technologies in the Manuscript Writing Process}
The authors used ChatGPT (OpenAI) to assist with improving the language and polishing the grammar.

\end{document}